\documentclass[12pt,reqno]{amsart}
\usepackage{amsmath,color,amssymb,amsthm}
\usepackage{enumerate}
\usepackage{graphicx}
\usepackage[margin=35mm]{geometry}
\usepackage[colorlinks=true,citecolor=black,linkcolor=black,urlcolor=blue,pdfauthor={Fiachra Knox, Bojan Mohar, David R. Wood},pdftitle={A golden ratio inequality on vertex degrees}]{hyperref}

\renewcommand{\leq}{\leqslant}
\renewcommand{\geq}{\geqslant}

\newtheorem{thm}{Theorem}
\newtheorem{lemma}[thm]{Lemma}

\DeclareMathOperator{\mad}{mad}

\newcommand{\N}{\mathbb{N}}

\newcommand{\eps}{\varepsilon}

\title[A golden ratio inequality for vertex degrees of graphs]{A golden ratio inequality\\ for vertex degrees of graphs}

\author{Fiachra Knox}
\address{\noindent Fiachra Knox\newline
Department of Mathematics\newline
Simon Fraser University\newline
Burnaby, Canada}
\email{fiachraknox@hotmail.com}
\thanks{F.K.~was supported by a PIMS Postdoctoral Fellowship.}

\author{Bojan Mohar}
\address{\noindent Bojan Mohar\newline
Department of Mathematics\newline
Simon Fraser University\newline
Burnaby, Canada}
\email{mohar@sfu.ca}
\thanks{B.M.~was supported in part by the NSERC Discovery Grant R611450 (Canada), by the Canada Research Chairs program, and by the Research Project J1-8130 of ARRS (Slovenia).}

\author{David R. Wood}
\address{\noindent David R. Wood\newline
School of Mathematical Sciences\newline
Monash University\newline
Melbourne, Australia}
\email{david.wood@monash.edu}
\thanks{Research of D.W.\ was supported by the Australian Research Council.}

\date{\today}

\begin{document}

\maketitle

\begin{abstract}
Motivated by the study of the crossing number of graphs, it is shown that, for trees, the sum of the products of the degrees of the end-vertices of all edges has an upper bound in terms of the sum of all vertex degrees to the power of $\phi^2$, where $\phi$ is the golden ratio. The exponent $\phi^2$ is best possible. This inequality is generalized for all graphs with bounded maximum average degree.
\end{abstract}

\bigskip
In a  study of the crossing number of graphs \cite{DKMW08,DKMW18}, the authors proved upper bounds on the crossing number for various graph classes. For a graph $G$ with vertex set $V(G)$ and edge set $E(G)$, if $d(v)$ denotes the degree of each vertex $v\in V(G)$, then these upper bounds are of the form 
\begin{equation*}
\alpha \sum_{uv \in E(G)}  d(u) d(v)
\qquad\text{or}\qquad
\alpha  \sum_{v \in V(G)}  d(v)^\beta,
\end{equation*}
where $\alpha$ and $\beta$ are constants depending on the particular class. 
As way to compare such bounds, the authors noted that 
\begin{align}
\label{Cubed}
\sum_{uv \in E(G)}  d(u) d(v) \;\leq\; \tfrac12  \sum_{v \in V(G)}  d(v)^3,
\end{align}
with equality for every regular graph (that is, if $d(u)=d(v)$ for all $u,v\in V(G)$). The proof is an easy exercise.

While the exponent of 3 in the right-hand side of \eqref{Cubed} cannot be improved for regular graphs,
for classes of graphs that allow for many different vertex degrees, such as trees and planar graphs,
it is natural to ask what is the minimum exponent such that every graph in the class satisfies an analogous inequality (allowing $\frac12$ to be replaced by some other constant).

%While the exponent of 3 in the right-hand side of \eqref{Cubed} cannot be improved for regular graphs,
%for classes of graphs that allow for many different vertex degrees, such as trees and planar graphs,
%it is natural to ask what is the minimum exponent $\beta$ such that for some constant $c$, every graph in the class satisfies the following inequality:
%\begin{align}
%\label{beta}
%\sum_{uv \in E(G)} d(u) d(v) \leq c \sum_{v \in V(G)} d(v)^\beta.
%\end{align}

%However, the authors of \cite{DKMW,DKMW08} actually desired an inequality similar to (\ref{Cubed}) where the third power on the right-hand side was  replaced by a square. Unfortunately, this is no longer true, even for simple families of graphs such as  trees. It would be of interest to lower the exponent 3 in such cases. It turns out that this is possible for trees.

We answer this question for trees and planar graphs. In fact, our result holds in a more general setting, which we now introduce. Let $\overline{d}(G)$ denote the average degree of a graph $G$. Note that for every tree $T$,
\begin{equation*}
    \overline{d}(T) = \frac{\sum_{v\in V(T)}d(v)}{|V(T)|} = \frac{2|E(T)|}{|V(T)|} = \frac{2(|V(T)|-1)}{|V(T)|} < 2.
\end{equation*}
Similarly, it is a simple consequence of Euler's formula that every planar graph has average degree less than 6. Note that every subgraph of a planar graph $G$ is also planar and thus its average degree is also less than 6. This motivates the following definition.

The \emph{maximum average degree} of a graph $G$ is the maximum of the average degrees of the (induced) subgraphs of $G$:
\begin{equation*}
    \mad(G) = \max_{H\subseteq G} \overline{d}(H).
\end{equation*}
Many well known classes of graphs have bounded maximum average degree: trees have maximum average degree less than $2$, series parallel graphs have maximum average degree less than $4$, planar graphs have maximum average degree less than $6$, and graphs of genus $g$ have maximum average degree $O(\sqrt{g})$.

It is therefore natural to ask whether  \eqref{Cubed} can be improved for classes of graphs with bounded maximum average degree. The following theorem answers this question. It is interesting and surprising that the golden ratio $\phi=1.618\ldots$ arises in this context.

\begin{thm}
\label{main}
For every $k \in \N$ and every graph $G$ with maximum average degree at most $2k$,
\begin{equation}
\label{eq:main}
  \sum_{uv \in E(G)} d(u) d(v) \;\leq\; k^{2-\phi} \sum_{v \in V(G)} d(v)^{\phi^2},
\end{equation}
where $\phi =\tfrac{1}{2}(1+\sqrt{5}\,)$. Moreover, both the exponent $\phi^2$ and the constant $k^{2-\phi}$ are best possible.
\end{thm}

For comparison with \eqref{Cubed}, note that $\phi^2=\phi+1=2.618\ldots.$  

\begin{proof}[Proof of \eqref{eq:main}.]
Our proof relies on the following special case of the weighted arithmetic mean--geometric mean inequality (see \cite[page~22]{Pachpatte} for example): for positive real numbers $x, y, p,q$  such that $p + q = 1$,
\begin{align}
\label{weightedAMGM}
x^p y^q \leq px + qy.
\end{align}

Hakimi~\cite{Hakimi65} proved that a graph $G$ has an orientation with maximum outdegree at most $k$ if and only if $G$ has maximum average degree at most $2k$. Fix such an orientation for $G$. For each arc $\overrightarrow{uv}$ of $G$, by \eqref{weightedAMGM} with $x = k^{-1} d(u)^{\phi^2}$, $y = d(v)^{\phi}$, $p = \phi^{-2} = 2-\phi$ and $q = \phi^{-1} = \phi-1$,
\begin{equation*}k^{\phi-2} d(u) d(v) \leq \phi^{-2} k^{-1} d(u)^{\phi^2} + \phi^{-1} d(v)^{\phi}.\end{equation*}
Summing over all arcs, and since $d^+(u)\leq k$,
\begin{align*}
& k^{\phi-2} \sum_{uv \in E(G)} d(u) d(v)\\
\leq\,& \Big( \sum_{u \in V(G)} \phi^{-2} k^{-1} d(u)^{\phi^2} d^+(u) \Big) + \Big(\sum_{v \in V(G)} \phi^{-1} d(v)^{\phi} d^-(v) \Big)\\
\leq\,& \phi^{-2} \Big(\sum_{u \in V(G)} d(u)^{\phi^2} \Big) + \phi^{-1} \Big(\sum_{v \in V(G)} d(v)^{\phi + 1} \Big)\\
=\,& \sum_{v \in V(G)} d(v)^{\phi^2}.\qedhere
\end{align*}
\end{proof}

To complete the proof of Theorem~\ref{main}, the following lemma shows that the exponent $\phi^2$ and the constant $k^{2-\phi}$ in \eqref{eq:main} cannot be improved. We use the following notation. For a real number $x$, let $\lceil x \rceil$ be the \emph{ceiling} of $x$; that is, the smallest integer greater than or equal to $x$. For a positive integer $t$, let $[t]$ denote the set $\{1,2,\dots,t\}$.

\begin{lemma}
For all $k\in\mathbb{N}$ and $\eps > 0$, there is a graph $G$ with maximum average degree at most $2k$ such that 
\begin{equation}
\label{goal}
(1 + \eps) \sum_{uv \in E(G)} d(u) d(v) \;\geq \; k^{2-\phi} \sum_{v \in V(G)} d(v)^{\phi^2} .
\end{equation}
\end{lemma}

%Fix $k\in\mathbb{N}$ and take any $\eps > 0$. Our goal is to construct a graph $G$ with maximum average degree at most $2k$ such that 
%\begin{equation}
%\label{goal}
%(1 + \eps) \sum_{uv \in E(G)} d(u) d(v) \;\geq \; k^{2-\phi} \sum_{v \in V(G)} d(v)^{\phi^2} .
%\end{equation}

\begin{proof}
First consider the case when $k = 1$. 
Choose integers $a$ and $R$ that are sufficiently large so that $R\geq 2$ and $a\geq 4$ and
$e^{3a^{-1}} (1 + (R-1)^{-1}) \leq  1 + \eps$.

Let $T$ be the tree with $R+1$ levels $L_0, \ldots, L_R$, where
$L_R$ consists of a single root vertex,
each vertex in $L_i$ has $\lceil a^{\phi^{i-1}}\rceil$ children in $L_{i-1}$ for $i\in[R]$,
and $L_0$ consists entirely of leaves, as illustrated in Figure~\ref{Tree}. 
\begin{figure}[h]
\centering
\includegraphics{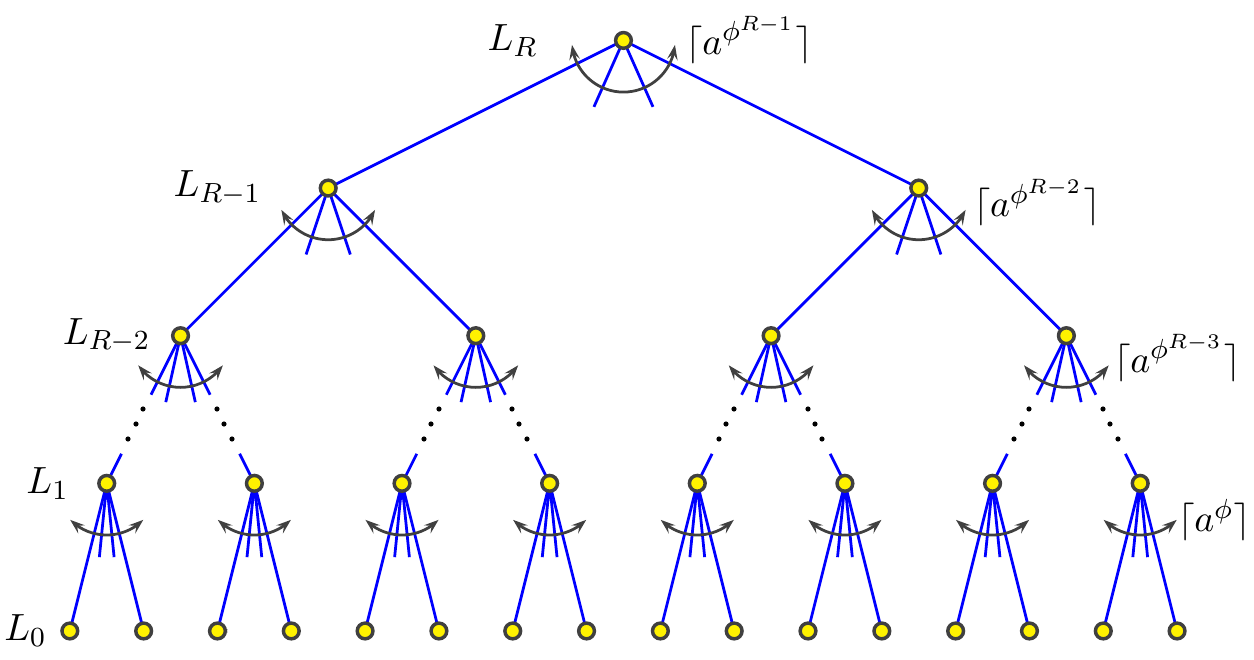}
\caption{The tree $T$. 
\label{Tree}}
\end{figure}
Note that for $i\in[R-1]$,
\begin{equation}
\label{LayerSize}
  |L_{i}| =
\lceil a^{\phi^{R-1}} \rceil
\lceil a^{\phi^{R-2}} \rceil
\cdots
\lceil a^{\phi^{i+1}} \rceil
\lceil a^{\phi^{i}} \rceil.
\end{equation}
We need a lower and an upper bound on $|L_i|$ with the ceilings removed. Note that the exponents in (\ref{LayerSize}) form a geometric progression and that
\begin{equation}
\label{eq:geometric exponents}
  \phi^{R-1} + \phi^{R-2} + \cdots + \phi^{i} + \phi^{i-1} = \frac{\phi^{R}-\phi^{i-1}}{\phi-1}
  = (\phi^{R}-\phi^{i-1})\phi = \phi^{R+1}-\phi^{i}.
\end{equation}
This gives the lower bound
\begin{equation*}
  |L_{i}| \geq a^{\phi^{R-1}} a^{\phi^{R-2}} \cdots a^{\phi^{i+1}} a^{\phi^{i}} = a^{\phi^{R+1}-\phi^{i}}.
\end{equation*}
To obtain an upper bound we use the inequality
\begin{equation*}
   a^t+1 = a^t (1+a^{-t}) \leq a^t\cdot e^{a^{-t}}.
\end{equation*}
We also use that $a^{-\phi} < \tfrac{1}{2}a^{-1}$ (since $a\geq 4$), which implies that
\begin{equation}
\label{eq:sum exponents of a}
  a^{-\phi^{R-1}} + a^{-\phi^{R-2}} + \cdots + a^{-\phi} + a^{-1} < 2a^{-1}.
\end{equation}
Now, for every $i\geq 0$,
\begin{align}
  |L_{i}|
  &\leq (a^{\phi^{R-1}}+1) (a^{\phi^{R-2}}+1) \cdots (a^{\phi^{i}}+1)\nonumber\\
  &\leq a^{\phi^{R+1}-\phi^{i}} \cdot\,  e^{a^{-\phi^{R-1}}} e^{a^{-\phi^{R-2}}} \cdots\, e^{a^{-\phi^{i}}}\nonumber \\
  &\leq a^{\phi^{R+1}-\phi^{i}} \cdot\,  e^{2a^{-1}},
  \label{eq:upper bound Li}
\end{align}
where the last inequality follows from \eqref{eq:sum exponents of a}.

Let $E_i$ be the set of edges of $T$ between $L_i$ and $L_{i-1}$ for each $i \in [R]$.
Note that $|E_i|=|L_{i-1}|$, and for each edge $uv\in E_i$ we have
$d(u)d(v) \geq a^{\phi^{i-1}} a^{\phi^{i-2}} = a^{\phi^{i-1} + \phi^{i-2}} = a^{\phi^i}$ if $i\ne1$
and  $d(u)d(v) = a+1$ if $i=1$ (since $R\geq 2$).

We obtain a lower bound for $\sum_{uv \in E(T)} d(u) d(v)$ as follows:
\begin{align}
\sum_{uv \in E(T)} d(u) d(v)
&= \sum_{i = 1}^R \sum_{uv \in E_i} d(u) d(v) \nonumber\\
&\geq |L_0| (a+1) + \sum_{i=2}^R |L_{i-1}| a^{\phi^{i}} \nonumber\\
&\geq a^{\phi^{R+1}-\phi} \cdot a
+ \sum_{i=2}^R a^{\phi^{R+1}-\phi^{i}} a^{\phi^i} \nonumber\\
& = a^{\phi^{R+1}} (a^{1-\phi} + R-1). \label{eq:lower bound proof}
\end{align}

On the other hand, we obtain an upper bound for $\sum_{v \in V(T)} d(v)^{\phi^2}$ as follows:
\begin{align*}
\sum_{v \in V(T)} d(v)^{\phi^2}
&= \sum_{i= 0}^R \sum_{v \in L_i} d(v)^{\phi^2} \\
&= \lceil a^{\phi^{R-1}}\rceil^{\phi^2} + |L_0| + \sum_{i=1}^{R-1} |L_i| \, \Bigl(\lceil a^{\phi^{i-1}}\rceil + 1 \Bigr)^{\phi^2}.
\end{align*}
The first term in the previous line is smaller than $a^{\phi^{R+1}}\cdot\, e^{a^{-1}}$. The second term has an upper bound given by (\ref{eq:upper bound Li}). Finally, each term in the remaining sum can be estimated in a similar way as \eqref{eq:upper bound Li}:
\begin{equation*}
 |L_i| \, \Bigl(\lceil a^{\phi^{i-1}}\rceil + 1 \Bigr)^{\phi^2}
 \leq |L_i| \, \bigl( a^{\phi^{i-1}} + 2 \bigr)^{\phi^2} \leq a^{\phi^{R+1}-\phi^{i-1}} \cdot\,  e^{3a^{-1}}.
\end{equation*}
This implies that
\begin{align*}
 \sum_{v \in V(T)} d(v)^{\phi^2}
& \leq e^{3a^{-1}} \left(1 + a^{-\phi^{-1}} + R-1\right) a^{\phi^{R+1}}\\
& \leq e^{3a^{-1}} \left(1 + \tfrac{1}{R-1}\right)\, \left(a^{-\phi^{-1}} + R-1\right) a^{\phi^{R+1}}.
\end{align*}
Hence, by (\ref{eq:lower bound proof}) and by the choice of $R$ and $a$,
\begin{equation*}
\sum_{v \in V(T)} d(v)^{\phi^2} \;\leq\; (1 + \eps) \sum_{uv \in E(T)} d(u) d(v).
\end{equation*}
This proves the lemma for $k=1$.

\medskip
To obtain the same result for higher $k$, simply take a blow-up $G$ of the tree $T$ defined above, in which each vertex is replaced by a stable set of $k$ vertices and each edge is replaced by a copy of the complete bipartite graph $K_{k,k}$, as illustrated in Figure~\ref{BlowUp}. 
\begin{figure}[h]
\centering
\includegraphics{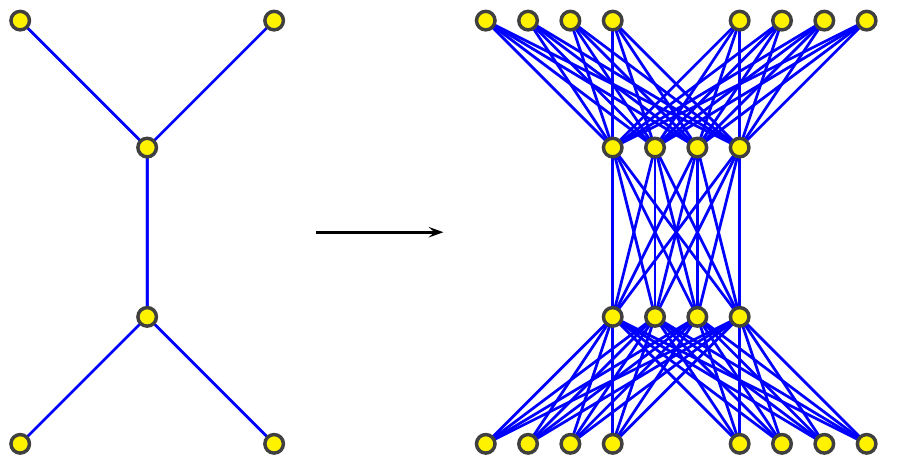}
\caption{Blow-up of a tree.
\label{BlowUp}}
\end{figure}
This construction multiplies all the degrees by a factor of $k$, and replaces each edge by $k^2$ edges. Thus, if $d_T(v)$ and $d_G(v)$ respectively denote the degree of a vertex $v$ in $T$ and in $G$, then 
\begin{align*}
k^{2-\phi} \sum_{v \in V(G)} d_G(v)^{\phi^2}
&= k^{3-\phi} \sum_{v \in V(T)} (kd_T(v))^{\phi^2}\\
&\leq (1 + \eps) k^{3-\phi} k^{\phi^2} \sum_{uv \in E(T)} d_T(u) d_T(v) \\
&= (1 + \eps)k^{2} \sum_{uv \in E(T)} (kd_T(u)) (kd_T(v)) \\
&= (1 + \eps) \sum_{uv \in E(G)} d_G(u) d_G(v),
\end{align*}
which proves \eqref{goal}. To see that $G$ has maximum average degree at most $2k$, orient each edge of $T$ towards the root, and then orient each edge of $G$ by following the orientation of the corresponding edge in $T$. Thus each vertex of $G$ has outdegree at most $k$, and $G$ has maximum average degree at most $2k$ by the above result of Hakimi~\cite{Hakimi65}.
\end{proof}

%This completes the proof of Theorem~\ref{main}. 

%\bibliographystyle{abbrv}
%\bibliography{Inequality}

\end{document}